\documentclass[psamsfonts,reqno,12pt]{amsart}

\usepackage{amssymb,amsfonts}
\usepackage[all,arc]{xy}
\usepackage{enumerate}
\usepackage{mathrsfs}
\usepackage{hyperref}
\usepackage{mathtools}
\usepackage{nicefrac}
\usepackage[margin=1.0in]{geometry} 
\usepackage{mathtools}
\usepackage{xcolor}

\newtheorem{thm}{Theorem}[section]
\newtheorem{cor}[thm]{Corollary}
\newtheorem{prop}[thm]{Proposition}
\newtheorem{lem}[thm]{Lemma}

\newtheorem{innercustomlem}{Lemma}
\newenvironment{customlem}[1]
{\renewcommand\theinnercustomlem{#1}\innercustomlem}
  {\endinnercustomlem}

\theoremstyle{definition}
\newtheorem{defn}[thm]{Definition}

\newtheorem{hyp}[thm]{Hypothesis}

\theoremstyle{remark}
\newtheorem{rem}[thm]{Remark}

\numberwithin{equation}{section}

\newcommand{\R}{\mathbb{R}}

\newcommand{\eps}{\varepsilon}

\newcommand{\inj}{\xhookrightarrow{{\kern 1em}}}

\usepackage[normalem]{ulem}

\title{Approximate control of the marked length spectrum by short geodesics}

\author{Karen Butt}

\begin{document}

\begin{abstract} 
The marked length spectrum (MLS) of a closed negatively curved manifold $(M, g)$ is known to determine the metric $g$ under various circumstances. We show that in these cases, 
(approximate) values of the MLS on a sufficiently large finite set approximately determine the metric.
Our approach is to recover the hypotheses of our main theorems in \cite{butt22}, namely multiplicative closeness of the MLS functions on the entire set of closed geodesics of $M$.  
We use mainly dynamical tools and arguments, but take great care to show the constants involved depend only on concrete geometric information about the given Riemannian metrics, such as the dimension, diameter, and sectional curvature bounds.
\end{abstract}

\maketitle


\section{Introduction} 
The marked length spectrum of a closed negatively curved Riemannian manifold $(M, g)$ is a function on the free homotopy classes of closed curves in $M$ which assigns to each class the length of its unique geodesic representative. 
It is conjectured that this function completely determines the metric $g$ up to isometry \cite{burnskatok}, and this is known under various conditions, namely in dimension 2 \cite{otalMLS, croke90}, in dimensions 3 or more when one of the metrics is locally symmetric \cite{ham99symplectic, BCGGAFA}, and in general when the metrics are sufficiently close in a suitable $C^k$ topology \cite{GL19, GKL22}.
(See the introductions to \cite{butt22, GL19} for more detailed discussions.)

In the case where $(M, g)$ has constant negative curvature and dimension at least 3, the fundamental group of $M$ already determines $g$ by Mostow Rigidty \cite{mostow}. 
When $\dim M = 2$, the Teichm{\"u}ller space of all such metrics has finite dimension $6 \, {\rm genus}(M) - 6$. 
Here it is known that the marked length spectrum on a sufficiently large finite subset 
determines the metric up to isometry. 
(See \cite[Theorem 10.7]{farbmargalit} and the introduction to \cite{ham2003}.) 

In the case of variable curvature, on the other hand, the space of all negatively curved metrics on $M$ is infinite-dimensional, so no finite set can suffice.
It is nevertheless natural to ask if finitely many closed geodesics can approximately determine the metric. As far as we know, this question has not been previously considered in the literature. 
In this paper, we do this in two cases: in dimension 2, and in dimension at least 3 when one of the metrics is locally symmetric. 
As mentioned above, these are two of the main cases where it is already known that the full marked length spectrum determines the metric up to isometry. 

In \cite{butt22}, we showed that in each of the above situations, two metrics are bi-Lipschitz-equivalent with constant close to 1 when the marked length spectra are multiplicatively close. 
In light of this, 
we answer our question by first proving that for arbitrary closed negatively curved manifolds, finitely many closed geodesics determine the full marked length spectrum approximately (Theorem \ref{mainthm}).
In fact, we do not require the lengths of the finitely many free homotopy classes to coincide exactly, but only approximately. 

\subsection{Statement of the main result}
To state our main result precisely, we first introduce some notation. 
Let $\mathcal{L}_g$ denote the marked length spectrum of $(M, g)$. Since the set of free homotopy classes of $M$ can be identified with conjugacy classes in the fundamental group $\Gamma$ of $M$, we will write $\mathcal{L}_g(\gamma)$ for the length of the geodesic representative of the conjugacy class of $\gamma \in \Gamma$ with respect to the metric $g$. 

If $(N, g_0)$ is another negatively curved Riemannian manifold with fundamental group isomorphic to $\Gamma$, there is a homotopy equivalence $f: M \to N$ inducing this isomorphism. 

Our work investigates what can be said about $g$ and $g_0$ satisfying 
the hypothesis below:

\begin{hyp}\label{MLSassfinite}
For $L > 0$, let
$ \Gamma_L := \{ \gamma \in \Gamma \, | \, \mathcal{L}_g(\gamma) \leq L \}$.
Now let $\eps > 0$ small and suppose 
\begin{equation*}
1 - \eps \leq \frac{\mathcal{L}_g(\gamma)}{\mathcal{L}_{g_0}(f_* \gamma)}  \leq 1 + \eps
\end{equation*} 
for all $\gamma \in \Gamma_L$.
\end{hyp}

If $L$ is sufficiently large, we obtain estimates for the ratio $\mathcal{L}_g/\mathcal{L}_{g_0}$ on all of $\Gamma$ in terms of $\eps$ and $L$ in Theorem \ref{mainthm} below.
Moreover, our estimates do not depend on the particular pair of metrics under consideration; they are uniform for all $(M, g)$ and $(N, g_0)$ satisfying Hypothesis \ref{MLSassfinite} with pinched sectional curvatures and diameters bounded above.


\begin{thm}\label{mainthm}
Let $(M, g)$ and $(N, g_0)$ be closed Riemannian manifolds of dimension $n \geq 3$, diameter at most $D$, and with sectional curvatures contained in the interval $[-\Lambda^2, -\lambda^2]$. 
Let $\mathcal{L}_g$ and $\mathcal{L}_{g_0}$ denote their marked length spectra. 
Let $\Gamma$ denote the fundamental group of $M$. 
Let $f: M \to N$ be a homotopy equivalence, and let $f_*$ denote the induced map on fundamental groups. 

Then there is $L_0 = L_0(n, \Gamma, \lambda, \Lambda)$ so that the following holds:
Suppose the marked length spectra $\mathcal{L}_g$ and $\mathcal{L}_{g_0} \circ f_*$ satisfy Hypothesis \ref{MLSassfinite} for 
some 
$\eps > 0$ and $L \geq L_0$. 
Let $K = \inf_{\gamma \in \Gamma} \frac{\mathcal{L}_{g}(\gamma)}{\mathcal{L}_{g_0}(f_* \gamma)}$.
Then, for all $\gamma \in \Gamma$, we have
\begin{equation*}
1 - (\eps + C L^{-\alpha}) \leq \frac{\mathcal{L}_g(\gamma)}{\mathcal{L}_{g_0}(f_* \gamma)} \leq 1 + (\eps + C L^{-\alpha})
\end{equation*} 
for some constant $C > 0$ depending only on $n, \Gamma, \lambda, \Lambda, D$ and any $\alpha < \frac{K}{2n} \frac{\lambda}{\Lambda}$. 
\end{thm}

\begin{rem}\label{rem:n=2}
Our proof of Theorem \ref{mainthm} gives the following result when $n = 2$. 
In this case, 
we may assume without loss of generality that the homotopy equivalence $f: M \to N$ is a diffeomorphism. 
This means there is some $A \geq 1$ so that $f$ and $f^{-1}$ are $A$-Lispchitz. Then Theorem \ref{mainthm} holds with the caveat that the constant $C$ in the conclusion depends additionally on $A$. 
This Lipschitz constant $A$ can equivalently be controlled by the $C^0$ distance $\Vert g - f^* g_0 \Vert_{C^0}$. (See also Remark \ref{rem:n=2fu} below.)
\end{rem}

\begin{rem}
The dependence of the constant $C$ on the diameters of $M$ and $N$ can be replaced with a dependence on the injectivity radius $i_M$ of $M$. See Remark \ref{rem:ivsd}. 
\end{rem}

\begin{rem}
One can obtain similar conclusions to those in Theorem \ref{mainthm} by combining Proposition \ref{holder}, the proof of which occupies the vast majority of this paper, with finite Livsic theorems such as  \cite[Theorem 1.2]{finitelivsic} and \cite{skatok}. However, our direct method in Section \ref{mainpfsection} yields estimates which depend only on concrete geometric information (dimension, diameter, sectional curvature) and not on the given flows;
see Remark \ref{finitelivsic}. 
\end{rem}

\subsection{Applications to rigidity}

Since the conclusion of Theorem \ref{mainthm} is the main hypothesis in \cite{butt22}, we recover versions of all our quantitative marked length spectrum rigidity results from only finite data, i.e., only assuming Hypothesis \ref{MLSassfinite}.
Let $\tilde \eps = \tilde \eps(\eps, L, n, \Gamma, \lambda, \Lambda, D) = \eps + CL^{-\alpha}$
 as in the conclusion of Theorem \ref{mainthm}.
This theorem states that  
Hypothesis \ref{MLSassfinite} results in
\begin{equation}\label{newMLSass}
 1 - \tilde \eps \leq \frac{\mathcal{L}_g(\gamma)}{\mathcal{L}_{g_0}(f_* \gamma)} \leq 1 + \tilde \eps 
 \end{equation}
for all $\gamma \in \Gamma$. 
Applying Theorems B and C in \cite{butt22} yields the following two corollaries. 
Note that the \emph{stable exponent} of $(M, g)$ refers to a constant $0 < \alpha_0 \leq 1$ such that the Anosov splitting of the geodesic flow on $T^1 M$ has $C^{\alpha_0}$ H{\"o}lder regularity. 
See \cite[Appendix, Proposition 4.4]{ballmann}, \cite[Lemma 2.16]{butt22} for more details. 
By \cite[Corollary 1.7]{hass94regularity}, we have $\alpha_0 = \min(1, 2 \lambda/\Lambda)$. 

\begin{cor}\label{maincor} 
Let $(M, g)$ and $(N, g_0)$ be a pair of homotopy-equivalent closed negatively curved Riemannian manifolds of dimension $n$ at least 3 and fundamental group $\Gamma$. 
Suppose further that $(N, g_0)$ is locally symmetric and that the curvature tensor $\mathcal{R}$ of $M$ satisfies $\Vert \nabla \mathcal{R} \Vert \leq R$ for some constant $R > 0$. 
Let $\alpha_0 = \alpha_0(\lambda, \Lambda) = \min(2 \lambda/\Lambda, 1)$ denote the above-defined stable H{\"o}lder exponent of $(M, g)$.
Then there exists $\eps_0 = \eps_0(\lambda, \Lambda, R)$ such that for any $\eps \in (0, \eps_0]$ 
the following holds: Suppose there is a 
homotopy equivalence
$f: M \to N$ so that the marked length spectra of $M$ and $N$ satisfy Hypothesis \ref{MLSassfinite} for $\eps$ as above. 
Then, there is a smooth diffeomorphism $F: M \to N$, homotopic to $f$, 
such that for all $v \in TM$ we have
\begin{equation}\label{lipest}
(1 - C \tilde{\eps}^{\alpha}) \Vert v \Vert_g \leq \Vert dF (v) \Vert_{g_0} \leq (1 + C \tilde{\eps}^{\alpha}) \Vert v \Vert_g
\end{equation}
for some $C > 0$ depending only on $n, \Gamma, D, i_0, a, b, R$ and for any  $\alpha < (1 - \eps) \frac{\lambda^2}{\Lambda} \frac{\alpha_0^2}{16 n}$.

\end{cor}

\begin{cor}\label{volumeest}
Let $(M, g)$ and $(N, g_0)$ be a pair of homotopy-equivalent closed Riemannian manifolds 
of dimension $n \geq 3$, fundamental group $\Gamma$, diameter at most $D$, injectivity radius at least $i_0$, and sectional curvatures contained in the interval $[- \Lambda^2, \lambda^2]$.
Suppose further that the geodesic flow on $T^1 N$ has $C^1$ Anosov splitting, and that the curvature tensor $\mathcal{R}$ of $(M, g)$ satisfies $\Vert \nabla \mathcal{R} \Vert \leq R$ for some $R > 0$. 
Let $\alpha_0 = \min(2\lambda/\Lambda, 1)$ denote the above-defined stable H{\"o}lder exponent of $M$. 

Then there exists $\eps_0 = \eps_0(\lambda, \Lambda, R)$ such that for any $\eps \in (0, \eps_0]$  the following holds: Suppose there is $f: M \to N$ so that the marked length spectra of $M$ and $N$ satisfy Hypothesis \ref{MLSassfinite} for $\eps$ as above.
Then
\[ (1 - C \tilde{\eps}^{\alpha}) {\rm Vol}(M) \leq {\rm Vol}(N) \leq (1 + C \tilde{\eps}^{\alpha}) {\rm Vol}(M) \]
for some constant $C = C(n, \Gamma, D,  i_0, \lambda, \Lambda, R) > 0$ and any $\alpha < (1 - \tilde \eps) \frac{\lambda^2}{\Lambda} {\alpha_0^2}$.
\end{cor}

When $n = 2$, note that the bound on the $C^0$ distance required to apply Theorem \ref{mainthm} in dimension 2 (see Remark \ref{rem:n=2}) is implied by the $C^{1, \beta}$ bound hypothesized in \cite[Theorem A]{butt22}.
Hence, combining \cite[Theorem A]{butt22} and Remark \ref{rem:n=2} yields the following.

\begin{cor}
Let $(M, g_0)$ be a closed surface with curvatures contained in the interval $[-\Lambda^2, - \lambda^2]$. Fix $0 < \beta < 1$ and $R > 0$ and let $\mathcal{U}_{g_0}^{1, \beta} (\lambda, \Lambda, R)$ denote the set of all $C^2$ metrics $g$ on $M$ satisfying $\Vert g - g_0 \Vert_{C^{1, \beta}(M)} \leq R$. 
Fix $B > 1$. Then there exists large enough $L = L(B, \lambda, \Lambda, R, \beta)$ so that for any pair $g, h \in \mathcal{U}_{g_0}^{1, \beta} (\lambda, \Lambda, R)$ whose marked length spectra satisfy  $ \mathcal{L}_g(\gamma) = \mathcal{L}_h(\gamma)$ for all $\{ \gamma \in \Gamma \, |, \mathcal{L}_g(\gamma) \leq L \}$,
there is a $B$-Lipschitz map $f: (M, g) \to (M, h)$. 
\end{cor}

\subsection{Structure of the paper}
In Section \ref{mainpfsection}, we start by stating the key dynamical facts used in our proof of Theorem \ref{mainthm}. 
Specifically, we use an estimate for the size of a covering of the unit tangent bundle $T^1 M$ by certain small ``flow boxes", 
in addition to a H{\"o}lder estimate for a certain orbit equivalence between the geodesic flows of $M$ and $N$. 
We then prove the theorem assuming these two facts. 
See the introduction to Section \ref{mainpfsection} below for a rough sketch of the argument.

The rest, and vast majority, of this paper is devoted 
to proving the above-mentioned covering lemma and H{\"o}lder estimate. 
The proofs rely on a few well-established consequences of the hyperbolicity of the geodesic flow. 
However, the standard results from the theory of Anosov flows (uniformly hyperbolic flows) are stated very generally and thus contain a multitude of constants which depend on the given flow in arguably mysterious ways. 
As a result, considerable technical difficulties arise in ensuring the constants depend only on select geometric and topological properties of $(M, g)$ and $(N, g_0)$. 

The main components of this analysis are as follows.
In Section \ref{localprodsection}, we use geometric arguments involving horospheres to investigate the local product structure of the geodesic flow, a key mechanism responsible for many of the salient features of hyperbolic dynamical systems. Indeed, the results of this section are used to prove both the covering lemma and the H{\"o}lder estimate.
The covering lemma is then quickly proved in Section \ref{coversection}.
Finally, in Section \ref{holdersection}, we prove the orbit equivalence of geodesic flows in \cite{gromov3rmks} is H{\"o}lder continuous, also with controlled constants.

\subsection*{Acknowledgements} I am very grateful to my advisor Ralf Spatzier for many helpful discussions and for help reviewing this paper.
I thank Vaughn Climenhaga, Dmitry Dolgopyat, Andrey Gogolev, and Thang Nguyen for helpful conversations. 
I would also like to thank Yair Minsky for help pointing out an error in a previously stated version of Proposition \ref{holder}. 
This research was supported in part by NSF grants DMS-2003712 and DMS-2402173.

\section{Proof of main theorem}\label{mainpfsection}
 
In this section, we will prove Theorem \ref{mainthm}
assuming two key statements: a covering lemma (Lemma \ref{coverr} below) and a H{\"o}lder estimate (Proposition \ref{holder} below). These statements are proved in Sections \ref{coversection} and \ref{holdersection}, respectively.

The basic idea is to start by covering the unit tangent bundle $T^1 M$ with finitely many sufficiently small ``flow boxes", that is, sets obtained by flowing local transversals 
for some small fixed time interval $(0,\delta)$. 
On the one hand, any periodic orbit of the flow that visits each of these boxes at most once is short, i.e., has period at most $\delta$ times the total number of boxes. 
On the other hand, any periodic orbit that is long, i.e., of length more than $\delta$ times the number of boxes, must return to at least one of the boxes more than once before it closes up.
In other words, long periodic orbits contain shorter almost-periodic segments. 
By the Anosov closing lemma, these are in turn shadowed by periodic orbits.
This allows us to approximate the lengths of long closed geodesics with sums of lengths of short ones.  
We then use a H{\"o}lder continuous orbit equivalence $\mathcal{F}: T^1 M \to T^1 N$ to argue that similar approximations hold for the corresponding closed geodesics in $N$.
From this, we are able to estimate the ratio of $\mathcal{L}_g(\gamma)/\mathcal{L}_{g_0}(\gamma)$ for all long geodesics $\gamma$ given our assumed estimate holds for short ones (Hypothesis \ref{MLSassfinite}).

We now introduce the precise statements of the aforementioned covering lemma and H{\"o}lder estimate.
Let $W^{si}$ for $i = s, u$ denote the strong stable and strong unstable foliations for the geodesic flow $\phi^t$ on the unit tangent bundle $T^1 M$.  
For $\delta > 0$, let $W^{si}_{\delta} (v)$ denote the connected component containing $v$ of  $W^{si}(v) \cap B(v, \delta)$, where $B(v, \delta)$ denotes a ball of radius $\delta$ in $T^1 M$ with respect to the Sasaki metric. 
(See Section \ref{localprodsection} for some background on the stable/unstable foliations and the Sasaki metric.)

Let 
$P(v, \delta) = \cup_{v' \in W_{\delta}^{ss}(v)} W_{\delta}^{su}(v')$
and let 
$R(v, \delta) = \cup_{t \in (- \delta/2, \delta/2)} \phi^t P(v, \delta)$. We will call $R(v, \delta)$ a $\delta$-\emph{rectangle}. 
For our proof of Theorem \ref{mainthm}, we use the following estimate for the number of $\delta$-rectangles needed to cover $T^1 M$.

\begin{lem}\label{coverr}
There is 
small enough $\delta_0 = \delta_0(n) > 0$, 
together with 
a constant $C =  C(n, \Gamma, \lambda, \Lambda)$
so that
for any $\delta \leq \delta_0$, there is a covering of $T^ 1 M$
by at most $C / \delta^{2n+1}$ $\delta$-rectangles.
\end{lem}

\begin{rem}
The main difficulty is showing that the constant $C$ does not depend on the metric $g$, but only on $n$, $\Gamma$, $\lambda$, $\Lambda$.
\end{rem}

\begin{rem}
Rectangles of the form $R(v, \delta)$ are often used to construct Markov partitions, e.g. in \cite{ratner1973markov}.
However, in Lemma \ref{coverr}, we are not constructing a partition, meaning we do not require the rectangles to be measurably disjoint. 
\end{rem}

Now consider the geodesic flows $\phi^t$ and $\psi^t$ on $T^1 M$ and $T^1 N$, respectively.
Recall that a homeomorphism $\mathcal{F}: T^1 M \to T^1 N$ is an \emph{orbit equivalence} if there is some function (cocycle) $a(t, v)$ so that 
 \[ \mathcal{F}(\phi^t v) = \psi^{a(t,v)} \mathcal{F}(v) \]
 for all $v \in T^1 M$ and for all $t \in \R$.  
Since $M$ and $N$ are homotopy-equivalent compact negatively curved manifolds, such an $\mathcal{F}$ exists by \cite{gromov3rmks}. 
Our proof of Theorem \ref{mainthm} relies on the following estimates for the regularity of $\mathcal{F}$.

 \begin{prop}\label{holder} 
Suppose $(M, g)$ and $(N, g_0)$ are closed Riemannian manifolds with sectional curvatures in the interval $[-\Lambda^2, -\lambda^2]$.
Suppose there is $A \geq 1$ such that $f: (M, g) \to (N, g_0)$ and $h: (N, g_0) \to (M, g)$ are $A$-Lipschitz homotopy equivalences with $f \circ h$ homotopic to the identity.
Let $f_*$ denote the map on fundamental groups induced by $f$, and let $K$ denote the number $\inf_{\gamma \in \Gamma} \frac{\mathcal{L}_g(\gamma)}{\mathcal{L}_{g_0}(f_* \gamma)}$.

Then there exists an orbit equivalence of geodesic flows $\mathcal{F}: T^1 M \to T^1 N$ which is $C^1$ along orbits and transversally H{\"o}lder continuous. 
More precisely, there is a small enough $\delta_0 = \delta_0(\lambda, \Lambda)$, together with a constant $C$, depending only on $\lambda$, $\Lambda$, ${\rm diam}(M)$, ${\rm diam}(N)$, $A$,
so that 
\begin{enumerate}
\item $d (\mathcal{F}(v), \mathcal{F}(\phi^t v)) \leq At$ for all $v \in T^1 \tilde M$ and all $t \in \R$,
\item $d(\mathcal{F}(v), \mathcal{F}(w)) \leq C d(v,w)^{\alpha}$ for all $v, w \in T^1 \tilde M$ with $d(v, w) < \delta_0$ and any $\alpha < K \frac{\lambda}{\Lambda}$.
\end{enumerate}
\end{prop}
 
 \begin{rem}\label{rem:n=2fu}
If $n \geq 3$, then the constant $A$ in the above statement depends only on $n, \Gamma, i_0, D, \Lambda$ by \cite[Proposition 2.42]{butt22}. 
 \end{rem}
 
\begin{rem}
It is a standard fact that any orbit equivalence of Anosov flows is $C^0$-close to a H{\"o}lder continuous one; in other words, there are constants $C$ and $\alpha$, depending on the given flows, i.e., on the metrics $g$ and $g_0$, so that $d(\mathcal{F}(v), \mathcal{F}(w)) \leq C d(v,w)^{\alpha}$ \cite[Theorem 6.4.3]{fisherhass}. 
However, we are claiming the stronger statement that for the orbit equivalence in \cite{gromov3rmks}, there is a uniform choice of $C$ and $\alpha$ for all $(M, g)$ and $(N, g_0)$ satisfying the hypotheses of the proposition. 
\end{rem}

\begin{rem}\label{rem:ivsd}
In our context, assuming
Hypothesis \ref{MLSassfinite},
the dependence of the constant $C$ on the diameters of $M$ and $N$ can be replaced with a dependence on the injectivity radius $i_M$. 
Indeed, by \cite[Section 0.3]{gromov1982volume}, the volume of $M$ is bounded above by a constant $V_0$ depending only on $n$, $\Gamma$, and $\lambda$. 
A standard argument 
then shows ${\rm diam}(M)$ is bounded above by some constant $D_0 = D_0(i_M, V_0, \Lambda)$.
Finally, since, for negatively curved manifolds, the injectivity radius is half the length of the shortest closed geodesic \cite[p.178]{petersen}, we have $(1 - \eps) i_M \leq i_N \leq (1 + \eps)i_M$, which shows ${\rm diam}(N)$ is bounded above by a constant $D_0 = D_0(i_M, V_0, \Lambda)$.
\end{rem}

To prove Theorem \ref{mainthm}, we start with a covering of $T^1 M$ by $\delta$-rectangles (see Lemma \ref{coverr}).
Let $\delta_0$ be as in Proposition \ref{holder}, then make $\delta_0$ smaller if necessary so that Lemma \ref{coverr} holds as well. This choice of $\delta_0$ depends only on $n, \lambda, \Lambda$. 
Now fix $\delta \leq \delta_0$, together with a covering $T^1 M = \cup_{i = 1}^m R(v_i, \delta)$. 
By Lemma \ref{coverr}, we can take $m \leq C / \delta^{2n+1}$. Since $\delta$ is now fixed, we use the notation $R_i$ for the rectangle $R(v_i, \delta)$ and $P_i$ for the transversal $P(v_i, \delta)$.

Let $v \in T^1 M$. Then $v \in P_i$ if and only if $\phi^t v \in R_i$ for all $t \in (-\delta/2, \delta/2)$. 
Moreover, if $v$ is tangent to a closed geodesic of length $\tau$, then for any rectangle $R_i$, the set 
\[\{ t \in (-\delta/2, \tau - \delta/2) | \, \phi^t v \cap R_i \neq \emptyset \}\]
is a (possibly empty) disjoint union of intervals of length $\delta$.

\begin{defn}
Fix a covering of $T^1M$ by $\delta$-rectangles $R_1, \dots, R_m$ as above. 
Suppose $\eta$ is a closed geodesic of length $\tau$ with $\eta'(0) = v$. Suppose that for each $i$, the set 
\[\{ t \in (-\delta/2, \tau - \delta/2) | \, \phi^t v \cap R_i \neq \emptyset \}\] 
consists of at most a single interval. Then we say $\eta$ is a \emph{short geodesic} (with respect to the covering $R_1, \dots R_m$).
\end{defn}

\begin{rem}\label{shortmeansshort}
Let $L = L(\delta) = C \delta^{-2n}$, where $C = C(n, \Gamma, \lambda, \Lambda)$ is the constant in the statement of Lemma \ref{coverr}.
If $\eta$ is a short geodesic, then $l_g(\eta) \leq m \delta \leq C \delta^{-2n} = L$.
\end{rem}

\begin{prop}\label{mcoding}
Let $\gamma$ be any closed geodesic in $M$. 
Then there is $k \in \mathbb{N}$ (depending on $\gamma$) and short geodesics $\eta_1, \dots, \eta_{k+1}$ 
so that 
\[ \left| l_g(\gamma) - \sum_{i=1}^{k+1} l_g(\eta_i) \right| <  2 k C \delta \]
for some constant $C = C(\lambda,\Lambda)$. 
\end{prop}

\begin{proof}
If $\gamma$ is already a short geodesic, then $k = 0$ and $\eta_1 = \gamma$. 
If not, then let $i$ be the smallest index so that $\gamma$ crosses through $R_i$ in at least two time intervals. 
Let $v \in P_i$ tangent to $\gamma$ and let $t_1 > 0$ be the first time so that $\phi^{t_1} v \in P_i$. 
By the Anosov closing lemma, there is $w_1$ tangent to a closed geodesic $\gamma_1$ of length $t_1'$ with $|t_1 - t_1'| < C \delta$, where $C$ depends only on the sectional curvature bounds $\lambda$ and $\Lambda$ 
(see Lemma \ref{ACL}). 
Similarly, applying the Anosov closing lemma to the orbit segment $\{ \phi^t v \, | \, t \in [t_1, l_g(\gamma)] \}$ gives $w_2$ 
tangent to a closed geodesic $\gamma_2$ of length $t_2'$ with $|(l_g(\gamma) - t_1) - t_2' | < C \delta$.  
This means $ |l_g(\gamma) - l_g(\gamma_1) - l_g(\gamma_2)| < 2 C \delta$.

Iterating the above process, we can ``decompose" $\gamma$ into short geodesics. More precisely,
if $\gamma_1$ is not a short geodesic, then there is some other rectangle $R_j$ through which $\gamma_1$ crosses twice. By the same argument as above, we get $|l_g(\gamma_1) - l_g(\gamma_{1,1}) - l_g(\gamma_{1,2})| < 2 C \delta$ for some $\gamma_{1, 1}, \gamma_{1, 2} \in \Gamma$. 
Continuing in this manner, we get the desired conclusion.
\end{proof}
 
Next, we show that $l_{g_0}(\gamma)$ is still well-approximated by the sum of the $g_0$-lengths of the \emph{same} free homotopy classes $\eta_1, \dots, \eta_{k+1}$ that were used to do the approximation with respect to $g$. For this, we use the estimates for the regularity of the orbit equivalence $\mathcal{F}: T^1 M \to T^1 N$ in Proposition \ref{holder}. 
Recall that $a(t,v)$ denotes the time-change cocycle, i.e., $\mathcal{F}(\phi^t v) = \psi^{a(t,v)} \mathcal{F}(v)$.

 \begin{lem}\label{ncoding}
Let $\gamma$ and $\eta_1, \dots, \eta_{k+1}$ as in Proposition \ref{mcoding}. Then an analogous estimate holds in $(N, g_0)$, namely,
\[ \left| l_{g_0}(\gamma) - \sum_{i=1}^{k+1} l_{g_0}(\eta_i) \right| < 2k C \delta^{\alpha}, \] where $C$ depends only on 
$\lambda$, $\Lambda$, $D$, $A$, 
and $\alpha = A^{-1} \lambda/\Lambda$ is the H{\"o}lder exponent in the statement of Proposition \ref{holder}. 
\end{lem}

\begin{proof}
As in the proof of Proposition \ref{mcoding}, let $v \in T^1M$ tangent to $\gamma$. 
By the Anosov closing lemma, there is $w_1 \in T^1 M$ tangent to a closed geodesic $\gamma_1$ of length $t_1'$ such that $d(v, w_1) < C \delta$, for some $C = C(\lambda, \Lambda)$ (Lemma \ref{ACL}). Additionally, $d(\phi^{t_1} v, \phi^{t_1'} w_1) < C \delta$. 

By Proposition \ref{holder}, we know $d(\mathcal{F}(v), \mathcal{F}(w_1)) < C \delta^{\alpha}$ for some $C = C(\lambda, \Lambda, D, A)$. 
Moreover, since $\mathcal{F}(v)$ and $\mathcal{F}(w_1)$ remain $C \delta^{\alpha}$-close after being flowed by times $a(t_1, v)$ and $a(t_1', w)$, respectively,
it follows that $|a(t_1, v) - a(t_1', w_1)| < 2 C \delta^{\alpha}$. (We defer the short proof of this fact to Section \ref{localprodsection}; see Lemma \ref{timesclose}.)

Similarly, the Anosov closing lemma applied to the orbit segment $\{ \phi^t v \, | \, t \in [t_1, l_g(\gamma)] \}$ gives $w_2$ tangent to a closed geodesic $\gamma_2$ of length $t_2'$.
By an analogous argument, $|a(l_g(\gamma) - t_1, v) - a(t_2', w_2)| < 2 C \delta^{\alpha}$.
Since $a(t,v)$ is a cocycle we get $|a(l_g(\gamma), v) - a(t_1', w_1) - a(t_2', w_2)| < 4C \delta^{\alpha}$. 

Using that $\mathcal{F}$ is a $\Gamma$-equivariant orbit-equivalence, it follows that $a(l_g(\gamma), v) = l_{g_0}(\gamma)$ whenever $v \in T^1 M$ is tangent to the closed geodesic $\gamma$. 
So the estimate in the previous paragraph can be rewritten as $|l_{g_0} (\gamma) - l_{g_0}(\gamma_2) - l_{g_0}(\gamma_2)| < 4 C \delta^{\alpha}$. 
As such, we can iterate the process in Proposition \ref{mcoding} and get an additive error of $4 C \delta^{\alpha}$ at each stage. 
\end{proof}
  
\begin{proof}[Proof of Theorem \ref{mainthm}]
Recall from Remark \ref{shortmeansshort} that $L = L(\delta) = C \delta^{-2n}$ for some $C = C(n, \Gamma, \lambda, \Lambda)$. Since we fixed $\delta \leq \delta_0 = \delta_0(n, \lambda, \Lambda)$, we see that $L \geq L_0 = L(\delta_0)$. 
By Lemma \ref{coverr}, 
this choice of $L_0$ depends only on $n$, $\Gamma$, $\lambda$, $\Lambda$, $i_M$.

Recall as well that we are assuming 
\begin{equation*}\label{hypothesis}
1 - \eps \leq \frac{\mathcal{L}_g(\gamma)}{\mathcal{L}_{g_0}(\gamma)} \leq 1 + \eps
\end{equation*}
for all $\gamma \in \Gamma_L: = \{ \gamma \in \Gamma \, | \, l_{g} (\gamma) \leq L \}$ (see Hypothesis \ref{MLSassfinite}). 
We then have
\begin{align*}
l_g(\gamma) &\leq \sum_{i=1}^{k+1} l_g(\gamma_i) + 2 k C \delta \tag{Proposition \ref{mcoding}} \\
&\leq  (1 + \eps) \sum_{i=1}^{k+1} l_{g_0}(\gamma_i) + 2 k C \delta \tag{Hypothesis \ref{MLSassfinite}} \\
&\leq ( 1 + \eps) l_{g_0}(\gamma) + (1 + \eps) 2k(2C' \delta^{\alpha} + C \delta) \tag{Proposition \ref{ncoding}} \\ 
&\leq (1 + \eps) l_{g_0} (\gamma) + k C'' \delta^{\alpha}.
\end{align*}
Using this, we consider the ratio
\begin{align*}
\frac{l_g(\gamma)}{l_{g_0}(\gamma)} &\leq (1 + \eps) + \frac{ k C'' \delta^{\alpha}}{l_{g_0}(\gamma)} \\
&\leq 1 + \eps + \frac{k C'' \delta^{\alpha}}{\sum_{i=1}^{k+1} l_{g_0}(\gamma_i) - 2 k \delta} \tag{Proposition \ref{ncoding}}\\
&\leq 1 + \eps + \frac{k C'' \delta^{\alpha} }{2 k i_N - 2 k \delta} \\
&= 1 + \eps +  \frac{C'' \delta^{\alpha}}{2 i_N - 2 \delta}.
\end{align*}
In the last inequality, we used the fact that $l_{g_0}(\gamma) \geq 2 i_N$ for all $\gamma$. 

Finally, by the definition of $L$ in Remark \ref{shortmeansshort}, we have $\delta = C L^{-1/2n}$, where $C$ is a constant depending only on $n$, $\Gamma$, $\lambda$, $\Lambda$. 
So we can write that the ratio $l_g(\gamma)/l_{g_0}(\gamma)$ is between $1 \pm (\eps + C' L^{-\alpha/2n})$, where $\alpha$ is the H{\"o}lder exponent in the statement of Proposition \ref{holder}. 
\end{proof}

\begin{rem}\label{finitelivsic}
There is a way to obtain approximate control of the marked length spectrum from finitely many geodesics by combining Proposition \ref{holder} with the finite Livsic theorem in \cite{finitelivsic}, but our direct method above yields better estimates. 

Let $a(t,v)$ denote the time change function for the orbit equivalence $\mathcal{F}$ in Proposition \ref{holder}. By the definition of $a(t,v)$ in Lemma \ref{alinj} 
this cocycle is differentiable in the $t$ direction. Let $a(v) = \frac{d}{dt}|_{t=0} a(t,v)$. It follows from the formula for $a(t,v)$ in Lemma \ref{alinj} and Lemma \ref{buniholder} that $a(v)$ is of $C^{\alpha}$ regularity, where $\alpha$ is the same H{\"o}lder exponent as in the statement of Proposition \ref{holder}. 
It follows from Lemma \ref{bbd} and the proof of Lemma \ref{alinj} that $\Vert a(v) \Vert_{C^0} \leq A$, where $A$ is the constant in Lemma \ref{bbd}. 
Hence, $\Vert a \Vert_{C^{\alpha}} \leq A + C$, where $C$ is the constant in Proposition \ref{holder}.

Now let $\{ \phi^t v \}_{0 \leq t \leq l_g(\gamma)}$ be the $g$-geodesic representative of the free homotopy class $\gamma$. Then $l_{g_0}(\gamma) = \int_0^{l_{g}(\gamma)} a(\phi^t v) \, dt.$
Let $f(v) = (a(v) - 1) / \Vert a - 1 \Vert_{C^{\alpha}}$. Then $\Vert f \Vert_{C^{\alpha}} \leq 1$ and Hypothesis \ref{MLSassfinite} implies 
\[ \frac{1}{l_g(\gamma)} \left|  \int_0^{l_{g}(\gamma)} f(\phi^t v) \, dt \right| \leq \frac{\eps}{A + C}\] 
for all $\gamma \in \Gamma_L$. 
Setting $L = \left( \frac{\eps}{C+A} \right)^{-1/2}$ means that $f$ satisfies the hypotheses of Theorem 1.2 in \cite{finitelivsic}. 
This theorem implies that for all $\gamma \in \Gamma$, the ratio $\mathcal{L}_g/\mathcal{L}_{g_0}$ is between $1 \pm C' \left( \frac{\eps}{C+A} \right)^{\tau}$, where $C'$ and $\tau$ are constants depending on the given flow. 
Our direct method above yields an exponent of $\alpha/4n$ in place of $\tau$. 
\end{rem}

\section{Local product structure}\label{localprodsection}

We consider the distance $d$ on $T^1 M$ induced by the \emph{Sasaki metric} $g^S$ on $T^1M$, which is in turn defined in terms of the Riemannian inner product $g$ on $M$ (see \cite[Exercise 3.2]{docarmo} for the definition). Throughout the rest of this paper, we will make use of the following standard facts relating the Sasaki distance $d$ to the distance $d_M$ on $M$ coming from the Riemannian metric $g$ and the distance $d_{T^1_q M}$ on $S^{n-1} \cong T^1_q M$.
Let $v, w \in T^1 M$ be unit tangent vectors with footpoints $p$ and $q$ respectively. Let $v' \in T^1_q M$ be the vector obtained by parallel transporting $v$ along the geodesic joining $p$ and $q$. 
Then we have
\begin{equation}\label{sasaki}
d_M(p, q), d_{T^1_q M} (v', w) \leq d(v,w) \leq d_M(p, q) + d_{T^1_q M} (v', w).
\end{equation}
For convenience, we will often write $d$ in place of $d_M$ when it is clear from context that we are considering the distance between points as opposed to between unit tangent vectors.

Recall the geodesic flow on the unit tangent bundle of a negatively curved manifold is Anosov, and thus has \emph{local product structure}. 
This means every point $v$ has a neighborhood $V$ which satisfies: for all $\eps > 0$, there is $\delta > 0$ so that whenever $x, y \in V$ with $d(x,y) \leq \delta$ there is a point $[x,y] \in V$ and a time $|\sigma(x, y)| < \eps$ such that
\[ [x,y] = W^{ss}_{\eps}(x) \cap W^{su}_{\eps}(\phi^{\sigma(x,y)} y) \]
\cite[Proposition 6.2.2]{fisherhass}.

Moreover, there is a constant $C_0 = C_0(\delta)$ so that $d(x,y) < \delta$ implies $d_{ss}(x, [x,y]), d_{su}(\phi^{\sigma(x,y)} [x,y], y) \leq C_0 d(x,y)$, where $d_{ss}$ and $d_{su}$ denote the distances along the strong stable and strong unstable manifolds, respectively. 

To describe the stable and unstable distances $d_{ss}$ and $d_{su}$, we first recall the stable and unstable manifolds $W^{ss}$ and $W^{su}$ for the geodesic flow have the following geometric description (see, for instance, \cite[p. 72]{ballmann}).  
Let $v \in T^1 \tilde M$. Let $p \in \tilde M$ be the footpoint of $v$ and let $\xi \in \partial \tilde M$ be the forward projection of $v \in T^1 \tilde M$ to the boundary. 
Let $B_{\xi, p}$ denote the Busemann function on $\tilde M$ and let $H_{\xi, p}$ denote its zero set. Then the lift of $W^{ss}(v)$ to $T^1 \tilde M$ is given by $\{ - {\rm grad} B_{\xi, p}(q) \, | \, q \in H_{\xi, p} \}$.
If $\eta$ denotes the projection of $-v$ to the boundary $\partial \tilde M$, then the lift of $W^{su}(v)$ to $T^1 \tilde M$ is analogously given by $\{ {\rm grad}B_{\eta, p}(q) \, | \, q \in H_{\eta, p} \}$. 

Now let $v \in T^1 \tilde M$ and $w \in W^{ss}(v)$. Let $p$ and $q$ denote the footpoints of $v$ and $w$ respectively. Define the stable distance $d_{ss}(v,w)$ to be the horospherical distance between $p$ and $q$, 
i.e., the distance obtained from restricting the Riemannian metric $g$ on $\tilde M$ to a given horosphere. The unstable distance is defined analogously. 

From the above description of $W^{ss}$ and $W^{su}$ in terms of normal fields to horospheres, it follows that the local product structure for the geodesic flow enjoys stronger properties than those for a general Anosov flow given in the first paragraph. 
First, the product structure is (almost) globally defined, meaning given $v \in T^1 \tilde M$, the neighborhood $V$ in the first paragraph consists of all vectors $w$ which do not lie on the oriented geodesic determined by $-v$ (see, for instance,
\cite{coudenelocprod}). 
Second, the bound on the temporal function $\sigma$ can be strengthened:
\begin{lem}\label{temporalbd}
If $d(v,w) < \delta$, then $|\sigma(v,w)| < \delta$,
for all $\delta > 0$.
\end{lem}

\begin{proof}
Let $p$ and $q$ denote the footpoints of $v$ and $w$ respectively. Then by (\ref{sasaki}), we know $d(p, q) < \delta$.
Let $\xi$ denote the forward boundary point of $v$ and let $\eta$ denote the backward boundary point of $w$. 
Let $p' \in H_{\xi, p}$ and $q' \in H_{\eta, q}$ be points on the geodesic through $\eta$ and $\xi$. 
Then $d(p', q')  = |\sigma(v, w)|$. Moreover, since the geodesic segment through $p'$ and $q'$ is orthogonal to both $H_{\xi, p}$ and $H_{\eta, q}$, it minimizes the distance between these horospheres. In other words, $|\sigma(v, w)| = d(p', q') \leq d(p, q) < \eps$. 
\end{proof}

This allows us to deduce the following key lemma, which was used in the proof of Proposition \ref{ncoding}.  

\begin{lem}\label{timesclose}
Consider the geodesic flow $\phi^t$ on the universal cover $T^1 \tilde M$. 
Suppose $d(v,w) < \delta_1$ and $d(\phi^s v, \phi^t w) < \delta_2$. Then $|s - t| < \delta_1 + \delta_2$. 
\end{lem}
 
\begin{proof}
Since $[\phi^s v, \phi^s w] = [\phi^s v, \phi^t w]$, we have
 \[ \phi^{\sigma(\phi^s v, \phi^s w)} \phi^s w = \phi^{\sigma(\phi^s v, \phi^t w)} \phi^t w. \]
Thus $\sigma(\phi^s v, \phi^s w) + s = \sigma(\phi^s v, \phi^t w) + t$. Rearranging gives 
\[ s - t = \sigma(\phi^s v, \phi^t w) - \sigma(\phi^s v, \phi^s w) = \sigma(\phi^s v, \phi^t w) - \sigma(v,w). \] By Lemma \ref{temporalbd}, the absolute value of the right hand side is bounded above by $\delta_1 + \delta_2$, which completes the proof. 
\end{proof}

Now assume $(M, g)$ has sectional curvatures between $-\Lambda^2$ and $-\lambda^2$. 
We will show the constant $C_0$ in the definition of local product structure can be taken to depend only on $\lambda$ and $\Lambda$, whereas \emph{a priori} it depends on the metric $g$. 
For our purposes, it will suffice to show the following proposition, which is formulated using the Sasaki distance $d$ between vectors in $T^1 M$ instead of the stable/unstable distances $d_{ss}$ and $d_{su}$ between vectors on the same horosphere. In fact, we will show later (Lemma \ref{horcompsas}) that the Sasaki distance $d$ between vectors on the same stable/unstable manifold is comparable to $d_{ss}$ and $d_{su}$, respectively.

\begin{prop}\label{localprodconst}
Suppose $(M, g)$ has sectional curvatures between $-\Lambda^2$ and $-\lambda^2$. 
Let $u \in T_p^1 \tilde M$.
Let $u_1 \in W^{ss}(u)$ and $u_2 \in W^{su}(u)$ with $d(u_1, u_2) \leq D$.
Let $d$ denote the distance in the Sasaki metric. Then there exists a constant $C_0 = C_0(\lambda, \Lambda, D)$ so that $d(u, u_i) \leq C_0 d(u_1, u_2)$.
\end{prop}

Our proof of Proposition \ref{localprodconst} relies on the geometry of horospheres, and we use many of the methods and results from the paper \cite{heintzehof} of the same title.
However, we additionally consider the Sasaki distances between unit tangent vectors in $T^1 \tilde M$ instead of just distances between points in $\tilde M$. 

Let $\xi \in \partial \tilde M$ and let $B = B_{\xi}$ be the associated Busemann function. Suppose $p \in \tilde M$ is such that $B(p) = 0$.
Let $v \in T_p^1 \tilde M$ perpendicular to ${\rm grad} B(p)$ and consider the geodesic $\gamma(s) = \exp_p(sv)$. 
Define $f(s) = B(\gamma(s))$. 
This is the distance from $\gamma(s)$ to the zero set of $B$. 
Moreover, $f'(s) = \langle {\rm grad}B, \gamma' \rangle = \cos \theta$, where $\theta$ is the angle between $\gamma'(s)$ and ${\rm grad} B (\gamma(s))$. 
In particular, $f'(0) = 0$.

\begin{lem}\label{Bupper}
For all $s \in \R$ we have $f(s) \leq \frac{\Lambda}{2} s^2$ and $\cos \theta(s) = f'(s) \leq \Lambda s$. 
\end{lem}

\begin{proof}
We have $f''(s) = \langle \nabla_{\gamma'} {\rm grad} B , \gamma' \rangle = \langle \nabla_{\gamma'_T} {\rm grad} B , \gamma'_T \rangle$, where $\gamma'_T$ denotes the component of $\gamma'$ which is tangent to the horosphere through $\xi$ and $\gamma(s)$. 
Note $\Vert \gamma'_T \Vert = \sin (\theta)$, where as before, $\theta$ is the angle between $\gamma'(s)$ and ${\rm grad} B (\gamma(s))$.

Thus $f''(s) = \langle J'(0), J(0) \rangle$, where $J$ is the stable Jacobi field along the geodesic through $\gamma(s)$ and $\xi$ with $J(0) = \gamma'_T(s)$.  (See, for instance, \cite[p.750--751]{BCGGAFA}.)
By \cite[Proposition IV.2.9 ii)]{ballmann}, we have $\Vert J'(0) \Vert \leq \Lambda \Vert J(0) \Vert$, which shows $f''(s) \leq \Lambda \Vert J(0) \Vert^2 \leq \Lambda$.  
Since $f(0)$ and $f'(0)$ are both 0, Taylor's theorem implies that for any $s$, there is $\tilde s \in [0,s]$ so that $f(s) = \frac{f''(\tilde s)}{2} s^2$. Thus, $f(s) \leq \frac{\Lambda}{2} s^2$ for all $s \geq 0$. 
Moreover, since $f'(0) = 0$, integrating $f''(s)$ shows $\cos \theta = f'(s) \leq \Lambda s$. 
\end{proof}

\begin{rem}\label{BupperREM}
We have $f''(s) = \langle \nabla_{\gamma'} {\rm grad}B, \gamma' \rangle = {\rm Hess} B(\gamma', \gamma'),$ 
and in the above proof, we have in particular shown that ${\rm Hess} B(\gamma', \gamma') \leq \Lambda$. In other words, ${\rm Hess} B(u, u) \leq \Lambda$ for any unit vector $u$. We will use this estimate in the proof of Lemma \ref{proj}.
\end{rem}

\begin{lem}\label{Blower}
Fix $S > 0$. Then there is a constant $c = c(\lambda, S)$ such that for all $s \in [0, S]$ we have $f(s) \geq \frac{c}{2} s^2$ and $\cos \theta = f'(s) \geq cs$. 
\end{lem}

\begin{proof}
As in \cite[Section 4]{heintzehof}, we use $f_{\lambda}(s)$ to denote the analogue of the function $f(s)$, but defined in the space of constant curvature $-\lambda^2$. 
By considering the appropriate comparison triangles, it follows that $f(s) \geq f_{\lambda}(s)$ and $f'(s) \geq f_{\lambda}'(s)$ \cite[Lemma 4.2]{heintzehof}. 
As in the proof of the previous lemma, we know $f_{\lambda}''(s) = \langle J(0), J'(0) \rangle$, where $\Vert J(0) \Vert = \sin \theta$. Solving the Jacobi equation explicitly in constant curvature gives $f_{\lambda}''(s) = \lambda \sin^2 \theta$. 
For all $s \in [0, S]$, this is bounded below by $\lambda \sin^2 \theta(S)$, which is a constant depending only on the value of $S$ and the space of constant curvature $-\lambda^2$.  
In other words, there is a constant $c = c(\lambda, S)$ so that $f_{\lambda}''(s) \geq c$ for all $s \in [0, S]$. 
As in the proof of the previous lemma, Taylor's theorem then implies $f_{\lambda}(s) \geq \frac{c}{2} s^2$, 
and integrating $f_{\lambda}''$ on the interval $[0,s]$ gives $f_{\lambda}'(s) \geq cs$. 
\end{proof}

\begin{rem}
From the above proof it is evident that $f_{\lambda}'(s)/s \to 0$ as $s \to \infty$, and as such the only way to get a positive lower bound for $\cos \theta/s$ is to restrict to a compact interval $[0, S]$, which will turn out to be sufficient for our purposes. 
\end{rem}

For the proofs of the next several lemmas, we will consider the following setup (see Figure 1 below). Let $u$ be a unit tangent vector with footpoint $p$ .
Let $v \in T^1_p M$ perpendicular to $u$ and let $\gamma(t) = \exp_p(tv)$. 
Fix $s > 0$ and let $u_1 \in W^{ss}(u)$ be such that
such that the geodesic determined by $u_1$ passes through $\gamma(s)$. 
Let $p_1$ denote the footpoint of $u_1$. Let $\eta$ denote the geodesic segment joining $p$ and $p_1$ and let $\alpha$ denote the angle this segment makes with the vector $u$. 
Let $q$ be the orthogonal projection of $p_1$ onto the geodesic $\gamma$. 
Consider the geodesic right triangle with vertices $p_1, q, \gamma(s)$. Let $\theta$ denote the angle at $\gamma(s)$ and let $\theta_1$ denote the angle at $p_1$.

\begin{figure}[h] 
\includegraphics[width=4in]{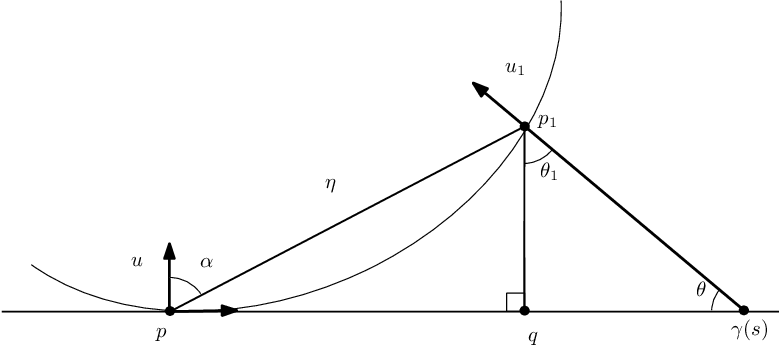}
\caption{}
\end{figure}

\begin{lem}\label{stableh}
Let $u_1 \in W^{ss}(u)$ as in Figure 1, and assume $s \leq S$ for some $S > 0$. 
Then there is a constant $C = C(\Lambda, S)$ so that
$d(u, u_1) \leq C s$. 
If $u_2 \in W^{su}(u)$, then $d(u, u_2) \leq Cs$ as well.
\end{lem}

\begin{proof}
Consider the setup in Figure 1. 
Let $\eta$ denote the geodesic joining $p$ and $p_1$ and let $P_{\eta}: T_p M \to T_{p_1} M$ denote parallel transport along this geodesic. Recall
\[ d (u, u_1) \leq d_M (p, p_1) + d_{T^1_{p_1} M} (Pu, u_1).\]
To bound $d_M(p, p_1)$, we use the triangle inequality, together with Lemma \ref{Bupper}:
\begin{align*}
d(p, p_1) &\leq d(p, q) + d(p_1, q)
\leq s+ d(p_1, \gamma(s)) 
\leq s + \Lambda s^2/2
\leq (1 + \Lambda S/2)s.
\end{align*}

To bound $d_{T^1_{p_1} M} (Pu, u_1)$, we first find bounds for the angles $\theta$ and $\theta_1$. 
We know from Lemma \ref{Bupper} that $\sin (\pi/2 - \theta) = \cos \theta \leq \Lambda s$. 
Moreover, $\sin(\pi/2 - \theta) \geq (2/ \pi) (\pi/2 - \theta)$ for $0 \leq \pi/2 \leq \theta$. 
Since the interior angles of geodesic triangles in $M$ sum to less than $\pi$, we know $\theta + \theta_1 < \pi/2$. Thus, $\theta_1 < \pi/2 - \theta \leq (\pi/2) \Lambda s$. 

Now let $\alpha$ denote the angle between $u$ and $\eta'$ at the point $p$. Then $\alpha$ is also the angle between $Pu$ and $\eta'$ at the point $p_1$, since parallel transport is an isometry and $\eta'$ is a geodesic. 
Since the angle sum of the geodesic triangle with vertices $p$, $p_1$ and $q$ is less than $\pi$, the angle in $T_{p_1} M$ between $\eta'$ and $[q, p_1]$ is strictly less than $\alpha$. 
Thus if we rotate $\eta'$ towards $Pu$, we must pass through the tangent vector to $[q, p_1]$ along the way.
Hence $d_{T_{p_1} M} (Pu, u_1) < \theta_1 \leq (\pi/2) \Lambda s$, which completes the proof of the upper bound for $d(u, u_1)$. The estimate for $d(u, u_2)$ follows by an analogous argument. 
\end{proof}

\begin{lem}\label{stableht}
Let $u \in T_p^1M$. Let $u_1 \in W^{ss}(u)$ be such that the footpoints $p$ and $p_1$ of $u$ and $u_1$ are distance $t$ apart.
Then
$d(u, u_1) \leq (1 + \Lambda)t$.
\end{lem}

\begin{proof}
Let $\eta$ denote the geodesic joining $p$ and $p_1$. 
Let $P_{\eta}: T_p M \to T_{p_1} M$ denote parallel transport along $\eta$.
Let $v_0 \in T_p^1M$ be the vector contained in the plane spanned by $u$ and $\eta'(0)$ so that $\langle u, v_0 \rangle = 0$ and $\langle \eta'(0), v_0 \rangle > 0$. 
Let $V(s)$ denote the parallel vector field along $\eta(s)$ with initial value $V(0) = v_0$. 
Let $\theta(s)$ be the angle between $V(s)$ and $- {\rm grad} B(\eta(s))$.  
Then $\theta_1 = \pi/2 - \theta(t)$ is the angle between $u_1$ and $P_{\eta} u$. 
We have 
\[ \sin(\pi/2 - \theta) = \cos(\theta) = \langle  V(t), - {\rm grad} B(\eta(t)) \rangle = \int_0^t \langle V(s), \nabla_{V} {\rm grad} B (\eta(s)) \rangle \, ds. \]
By the same argument as in the proof of Lemma \ref{Bupper}, this integral is bounded above by $\Lambda t$. 
Hence, $d(u, u_1) \leq d_M(p,p_1) + d_{T_{p_1} M}(P_{\eta}u, u_1) \leq t + \Lambda t$. 
\end{proof}

\begin{proof}[Proof of Proposition \ref{localprodconst}]
It suffices to show the statement for $i = 1$.
Consider the hypersurface formed by taking the exponential image of ${\rm grad}B(p)^{\perp}$.
Let $s_0 = s_0(\Lambda) > 0$ sufficiently small so that $1 - \Lambda s_0^2 \geq 1/4$ and  $1 - \Lambda s_0 \geq 1/4$.
We will consider two separate cases. 

The first case is where the geodesic determined by $u_1$ intersects the above hypersurface and the point of intersection is distance at most $s_0$ away from $p$. 
Let $x$ denote the point on this hypersurface which is on the geodesic determined by $u_1$. 
Let $v_1 \in T_p^1 M$ perpendicular to $u$ such that $\gamma(s) := \exp_p(s v_1) = x$, where $s = d(p, x_1) \leq s_0$.
By Lemma \ref{stableh}, we have $d(u, u_1) \leq Cs$ for some $C = C(\Lambda, {\rm diam}(M))$. So it suffices to bound $d(u_1, u_2)/s$ from below by some constant depending only on the desired parameters. 
To do so, let $p_1$ and $p_2$ denote the footpoints of $u_1$ and $u_2$, and let $q_1, q_2$ denote the orthogonal projections of $p_1, p_2$ onto the tangent plane ${\rm grad} B(p)^{\perp}$. 
We now note that $D \geq d(u_1, u_2) \geq d(p_1, p_2) \geq d(p_1, q_1)$.

To bound $d(p_1, q_1)$ from below, we consider a comparison right triangle in the space of constant curvature $- \lambda^2$ with hypotenuse equal to $d(\gamma(s), p_1) = f(s)$ and angle $\theta$ equal to the angle between ${\rm grad} B$ and $\gamma'$ at the point $\gamma(s) = x$. Let $l$ denote the length of the side opposite to the angle $\theta$. 
Then, \cite[Theorem 7.11.2 ii)]{beardon} and Lemma \ref{Blower} give
\[ \sinh(d(p_1, q_1)) \geq \sinh(l) = \sin(\theta (s)) \sinh( f(s) ) \geq \sin ( \theta(s_0)) \sinh(cs^2). \]
To bound $\sin (\theta(s_0))$ from below, we use the upper bound for $\cos \theta$ from Lemma \ref{Bupper}; this gives $\sin^2(\theta(s_0)) \geq 1 - \Lambda s_0^2 \geq 1/4$ by choice of $s_0$. 
This completes the proof in this case.

We now proceed to the second case.
Here, we have $d(p, q_1) \geq s_0 - \Lambda s_0^2 \geq 1/4 s_0$. 
Moreoever, $d(p_1, q_1) \geq c s_0^2$ by the proof of the first case.
Since $M$ is negatively curved, and hence CAT(0), we have $t_1:= d(p, p_1) \geq \sqrt{(1/4 s_0)^2 + (c s_0^2)^2}$, ie, $t_1 \geq c_1$ for some $c_1 = c_1(\Lambda)$. 
To relate $d(u_1, u_2)$ and $d(u_1, u)$, we will, as before, use that $d(u_1, u_2) \geq d(p_1, q_1)$. 
Since by Lemma \ref{stableht}, $(1 + \Lambda)t_1 \geq d(u, u_1)$, it suffices to bound from below the ratio $d(p_1, q_1)/t_1$.

For this, let $\beta$ denote the angle between this hypersurface and the geodesic $\eta(t)$ joining the footpoints of $u$ and $u_1$. 
Consider a comparison right triangle in the space of constant curvature $- \lambda^2$ with hypotenuse equal to $t_1$ and angle $\beta$.
Let $l$ denote the length of the side opposite to the angle $\beta$. Then, using the fact that triangles in $M$ are thinner than this comparison triangle, together with \cite[Theorem 7.11.2 ii)]{beardon}, gives
\[
\sinh(d(p_1, q_1)) \geq \sinh(l) = \sin(\beta) \sinh(t_1). 
\]
Since for all $t \in [0, D]$ we have $t \leq \sinh(t) \leq Ct$ for some $C = C(D)$, it remains to bound $\sin \beta$ from below.

Let $P_t (u)$ denote the parallel transport along the geodesic $\eta$ of $u$ from $p$ to $\eta(t)$. 
We then have
\begin{equation}
\cos \beta = \langle P_{t_0} (u), {\rm grad} B(\eta(t_1)) \rangle = 1 + \int_0^{t_1} \langle P_{t_1} (u), \nabla_{\eta'(t_1)} {\rm grad} B(\eta(t_1)) \rangle
\end{equation} 
By \cite[Corollary 4.2]{BrinKarcher}, the integrand is uniformly bounded in absolute value by some constant $C = C(\lambda, \Lambda)$. This shows $\sin \beta \geq C t_1 \geq C c_1(\Lambda)$, which completes the proof.
\end{proof}

Proposition \ref{localprodconst} allows us to deduce the following refinement of the Anosov Closing Lemma, where we can say the constants involved depend only on concrete geometric information about $(M, g)$, namely the diameter and the sectional curvature bounds. 
Note that now the setting is $T^1 M$ as opposed to the universal cover $T^1 \tilde M$.

\begin{lem}\label{ACL}
Fix $\delta > 0$ and suppose $v, \phi^t v \in T^1 M$ are such that $d(v, \phi^t v) < \delta$. 
Then either $v$ and $\phi^t v$ are on the same local flow line or there is $w$ with $d(v,w) < C \delta$ so that $w$ is tangent to a closed geodesic of length $t' \in [t - C \delta, t + C \delta]$, 
where $C$ is a constant depending \emph{only} on the sectional curvature bounds $\lambda$ and $\Lambda$. 
\end{lem}

\begin{proof}
The proof of the usual Anosov Closing Lemma in \cite[Figure 2]{frankelACL} (see also \cite[3.6, 3.8]{bowenbook}) shows the constant $C$ depends only on the local product structure constant $C_0$. By Proposition \ref{localprodconst}, we know this depends only on $\lambda$ and $\Lambda$. 
\end{proof}

\section{Covering lemma}\label{coversection}

In this section, we prove the following covering lemma, which was one of the key statements we used in the proof of the main theorem. 

\begin{customlem}{\ref{coverr}}
There is 
small enough $\delta_0 = \delta_0(n) > 0$, 
together with 
a constant $C =  C(n, \Gamma, \lambda, \Lambda)$
so that
for any $\delta \leq \delta_0$, there is a covering of $T^ 1 M$
by at most $C / \delta^{2n+1}$ $\delta$-rectangles.
\end{customlem}

We start with a preliminary lemma. 
\begin{lem}\label{sasball}
Let $B(v, \delta)$ be a ball of radius $\delta$ in $T^1 M$ with respect to the Sasaki metric. 
There is small enough $\delta_0$, depending only on the dimension $n$, so that for all $\delta < \delta_0$ we have
${\rm vol} (B(v, \delta)) \geq c \delta^{2n + 1}$
for some constant $c = c(n)$. 
\end{lem}

\begin{proof}
First we claim $B(v, \delta) \supset B_M(p, \delta/2) \times B_{S^{n-1}}(v, \delta/2)$, where $B_M(p, \delta/2)$ is a ball of radius $\delta/2$ in $M$ and $B_{S^{n-1}}(v, \delta/2)$ is a ball of radius $\delta/2$ in the unit tangent sphere $T_p^1 M$. 
This follows immediately from (\ref{sasaki}). 
Since $M$ is negatively curved, Theorem 3.101 ii) in \cite{gallot} implies ${\rm vol} B_M(p, \delta/2) \geq \beta_n \delta^n /2^n$, where $\beta_n$ is the volume of the unit ball in $\R^n$. 
By Theorem 3.98 in \cite{gallot}, we have ${\rm vol} B_{S^{n-1}}(v, \delta/2) = \frac{\beta_{n-1}  \delta^{n-1}}{2^{n-1}} (1 - \frac{n-1}{6(n+1)} \delta^2 + o(\delta^4))$.
Then for $\delta$ less than some small enough $\delta_0$, we can write 
\[ B_{S^{n-1}}(v, \delta/2) \geq \frac{\beta_{n-1} \delta^{n-1}}{2^{n-1}}  \left( 1 - 2 \frac{n-1}{6(n+1)} \delta^2 \right) \geq c \delta^{n+1}, \]
for some $c = c(n)$.
The quantity $\delta_0$ depends only on the coefficients of the Taylor expansion of ${\rm vol} B_{S^{n-1}}(v, \delta/2)$, which depend only on the geometry of $S^{n-1}$. So we can say $\delta_0$ depends only on $n$. 
Therefore, the volume of the Sasaki ball $B(v, \delta)$ is bounded below by $c \delta^{2n + 1}$ for some other constant $c = c(n)$ depending only on $n$.
\end{proof}

\begin{proof}[Proof of Lemma \ref{coverr}]
Let $C$ as in Proposition \ref{localprodconst} and $\delta_0$ as in Lemma \ref{sasball}. Let $c = 1/C$ and let $\delta < \delta_0/2c$. 
Let $v_1, \dots, v_m$ be a maximal $c \delta$-separated set in $T^1 M$ with respect to the Sasaki metric. 
We claim that the balls $B(v_1, c \delta), \dots , B(v_m, c \delta)$ cover $T^1 M$.
If not, there is some $v$ such that $d(v, v_i) \geq c \delta$ for all $i$. This contradicts the fact that $v_1, \dots, v_m$ was chosen to be a \emph{maximal} $c \delta$-separated set. 

This implies that the rectangles $R(v_1, \delta) \dots R(v_m, \delta)$ cover $T^1 M$ as well. 
Indeed, let $w \in B(v, c \delta)$. Then by Lemma \ref{temporalbd} there is a time $\sigma = \sigma(v,w) < c \delta$ and a point $[v, w] \in T^1 M$ so that $[v,w] = W^{ss}(v) \cap W^{su}(\phi^{\sigma} w)$. Thus $d(v, \phi^{\sigma} w) \leq \delta_0$ and Proposition \ref{localprodconst} implies $d_ss(v, [v,w]), d_{su}([v,w], \phi^{\sigma} w) < C c\delta = \delta$ as desired. 

Now we estimate $m$. 
Since $v_1, \dots, v_m$ if $c \delta$-separated, it follows that for $i \neq j$ we have
$B(v_i, c \, \delta/2) \cap B(v_j, c \, \delta/2) = \emptyset$. 
Hence  
\[ m \inf_i {\rm vol} (B(v_i, c \delta/2)) \leq {\rm vol}(T^1 M) = {\rm vol}(S^{n-1}) {\rm vol}(M). \] 
By  \cite[0.3 Thurston's Theorem]{gromov1982volume}, we have ${\rm vol}(M)$ is bounded above by a constant depending only on $n$, $\Gamma$ and the upper sectional curvature bound $-\lambda^2$. 
This, together with
Lemma \ref{sasball}, gives $m \leq C / \delta^{2n + 1}$ for some constant $C = C(n, \Gamma, \lambda, \Lambda )$.
 \end{proof}

\section{H{\"o}lder estimate}\label{holdersection}

In this section, we prove Proposition \ref{holder}, which was one of the main ingredients in the proof of Theorem \ref{mainthm}. 
In light of the methods in \cite[Section 2.3]{butt22}, it suffices to show the following statement.

\begin{prop}\label{holder2}
Suppose $(M, g)$ and $(N, g_0)$ are closed Riemannian manifolds with sectional curvatures in the interval $[-\Lambda^2, -\lambda^2]$.
Suppose there is $A \geq 1$ such that $f: (M, g) \to (N, g_0)$ and $h: (N, g_0) \to (M, g)$ are $A$-Lipschitz homotopy equivalences with $f \circ h$ homotopic to the identity.

Then there exists an orbit equivalence of geodesic flows $\mathcal{F}: T^1 M \to T^1 N$ which is $C^1$ along orbits and transversally H{\"o}lder continuous. 
More precisely, there is a small enough $\delta_0 = \delta_0(\lambda, \Lambda)$, together with a constant $C$, depending only on $\lambda$, $\Lambda$, ${\rm diam}(M)$, ${\rm diam}(N)$, $A$,
so that 
\begin{enumerate}
\item $d (\mathcal{F}(v), \mathcal{F}(\phi^t v)) \leq At$ for all $v \in T^1 \tilde M$ and all $t \in \R$,
\item\label{holdexp} $d(\mathcal{F}(v), \mathcal{F}(w)) \leq C d(v,w)^{A^{-1} \lambda/\Lambda}$ for all $v, w \in T^1 \tilde M$ with $d(v, w) < \delta_0$.
\end{enumerate}
 \end{prop}

\begin{rem}
In \cite[Theorem 2.38]{butt22}, we show the H{\"o}lder exponent $A^{-1} \lambda/ \Lambda$ in part \ref{holdexp} can be replaced with $\inf_{\gamma \in \Gamma} \frac{\mathcal{L}_g(\gamma)}{\mathcal{L}_{g_0}(f_* \gamma)} \frac{\lambda}{\Lambda}$. This yields Proposition \ref{holder}. 
\end{rem}


We first show that the homotopy equivalence $f: \tilde M \to \tilde N$ is a quasi-isometry with controlled constants. 

\begin{lem}\label{lem:QI}
Let $(M, g)$ and $(N, g_0)$ be closed negatively curved Riemannian manifolds, and suppose there is $A \geq 1$ such that $f: (M, g) \to (N, g_0)$ and $h: (N, g_0) \to (M, g)$ are $A$-Lipschitz homotopy equivalences with $f \circ h$ homotopic to the identity.
Let $\tilde f: \tilde M \to \tilde N$ be a lift of $f$.
Then there is a constant $B = B(A, {\rm diam}(M), {\rm diam}(N))$ such that 
\[
A^{-1} \, d_g(x_1, x_2) - B \leq d_{g_0}(\tilde f(x_1), \tilde f(x_2)) \leq A \, d_g (x_1, x_2).
\]
for all $x_1, x_2 \in \tilde M$
\end{lem}

\begin{proof}

Given that $f$ is $A$-Lipschitz, we only need to show the first inequality. We follow the approach of \cite[Proposition C.12]{benedetti12}, but we need to show $B$ depends only on the desired parameters.

First, consider the following fundamental domain $D_M$ for the action of $\Gamma$ on $\tilde M$ (see \cite[Proposition C.1.3]{benedetti12}). 
Fix $p \in \tilde M$. 
Let 
\begin{equation}\label{dirfundom}
D_M = \{ x \in \tilde M \, | \, d(x, p) \leq d(x, \gamma.p) \, \forall \gamma \in \Gamma \}.
\end{equation}
We claim the diameter of $D_M$ is bounded above by $2 \, {\rm diam}(M)$.
Indeed, let $x \in \tilde M$ so that $d(p, x) > {\rm diam}(M)$. This means there is some $\gamma \in \Gamma$ so that $d(x, \gamma.p) < d(x,p)$. 
In other words, any geodesic in $\tilde M$ starting at $p$ stays in $D_M$ for a time of at most ${\rm diam}(M)$.  
So if $x_1, x_2 \in D_M$ so that $d(x_1,x_2) = {\rm diam}(D_M)$, then $d(x_1, x_2) \leq d(x_1, p) + d(x_2, p) \leq 2 \, {\rm diam}(M)$, which proves the claim. 

Next we claim that for all $x \in \tilde M$, we have $d(h\circ f(x), x) \leq 2 (1 + A^2) {\rm diam}(M)$.
Indeed, since $h$ and $f$ are both continuous and $\Gamma$-equivariant, so is $h \circ f$,
and thus it suffices to check the statement for $x$ in a compact fundamental domain $D_M$. 
Since $f$ and $h$ are $A$-Lipschitz, it follows that the function $x \mapsto d(h \circ f(x), x)$ is $(1 + A^2)$-Lipschitz:
\begin{align*}
|d(h \circ f (x), x) - d(h \circ f(y), y)| \leq d(h \circ f(x), h \circ f(y)) + d(x, y) \leq (1 + A^2) d(x,y).
\end{align*}
Noting $d(x,y) \leq D_M \leq 2 \, {\rm diam}(M)$ proves the claim.

To complete the proof, we can now use the argument in \cite{benedetti12} verbatim. 
By the previous claim, we obtain
\[ d(h(f(x_1)), h(f(x_2)) \geq d(x_1, x_2) - 4 (1 + A^2) {\rm diam}(M). \]
Then, the Lipschitz bounds for $f$ and $h$ give
\[ d(f(x_1), f(x_2)) \geq A^{-1} d(h \circ f (x_1), h \circ f(x_2)) \geq A^{-1} \big(d(x_1, x_2) - 4 (1 + A^2) {\rm diam}(M)\big),\]
which completes the proof.  
\end{proof}

We will repeatedly use the following fact about quasi-geodesics remaining bounded distance away from their corresponding geodesics:

\begin{lem}[Theorem III.H.1.7 of \cite{BH13nonpos}]\label{Morselemma}
Let $\tilde f$ be the quasi-isometry from Lemma \ref{lem:QI}.
Let $c(t)$ be any geodesic in $\tilde M$ and let $\eta$ be its corresponding geodesic in $\tilde N$ obtained from the boundary map $\overline{f}: \partial \tilde M \to \partial \tilde N$. 
Then there is a constant $R$, depending only on the quasi-isometry constants $A$ and $B$ of $\tilde f$ and the upper sectional curvature bound $-\lambda^2$ for $N$, 
so that 
$d(\tilde f(c(t)), P_{\eta}(\tilde f(c(t))) \leq R$ for any $t \in \R$.  
\end{lem}

To prove Proposition \ref{holder}, we will take $\mathcal{F}$ to be the map in \cite{gromov3rmks}, whose construction we now recall. 
Let $\eta$ be a bi-infinite geodesic in $\tilde M$ and let $\zeta = \overline{f} (\eta)$ be the corresponding geodesic in $\tilde N$, where $\overline{f} : \partial^2 \tilde M \to \partial^2 \tilde N$ is obtained from extending the quasi-isometry $\tilde f: \tilde M \to \tilde N$ to a map $\partial \tilde M \to \partial \tilde N$.
 
Let $P_{\zeta}: \tilde N \to \zeta$ denote the orthogonal projection. Note this projection is $\Gamma$-equivariant, i.e., $\gamma P_{\zeta} (x) =  P_{\gamma \zeta} (\gamma x)$. 
If $(p, v) \in T^1 \tilde M$ is tangent to $\eta$, define ${\mathcal F}_0(p, v)$ to be the tangent vector to $\zeta$ at the point $P_{\zeta} \circ \tilde f (p)$. 
Thus ${\mathcal F}_0: T^1 \tilde M \to T^1 \tilde N$ is a $\Gamma$-equivariant map which sends geodesics to geodesics. 
As such, we can define a cocycle $b(t, v)$ to be the time which satisfies
\begin{equation}\label{F0}
{\mathcal F}_0 (\phi^t v) = \psi^{b(t,v)} {\mathcal F}_0(v).
\end{equation}

It is possible for a fiber of the orthogonal projection map to intersect the quasi-geodesic $\tilde f(\eta)$ in more than one point; thus, ${\mathcal F}_0$ is not necessarily injective.
In order to obtain an injective orbit equivalence, we follow the method in \cite{gromov3rmks} and average the function $b(t, v)$ along geodesics. 
We include a detailed proof for completeness. 

\begin{lem}\label{alinj}
Let 
\[ a_l (t, v) = \frac{1}{l} \int_t^{t+l} b(s, v) \, ds. \] 
There is a large enough $l$ 
so that $t \mapsto a_l(t,v)$ is injective for all $v$. 
\end{lem}

\begin{proof}
The fundamental theorem of calculus gives 
\begin{equation}\label{al}
\frac{d}{dt}  a_l(t, v) = \frac{b(t+l,v) - b(t,v)}{l}.
\end{equation}
We claim there is a large enough $l$ so that this quantity is always positive. To this end, suppose $b(t+l,v) - b(t,v) = 0$. 
This means ${\mathcal F}_0( \phi^t v)$ and ${\mathcal F}_0( \phi^{t + l}v)$ are in the same fiber of the normal projection onto the geodesic $\overline f(v)$.
Since $s \mapsto \tilde f(\phi^s v)$ is a quasi-geodesic, by Lemma \ref{Morselemma}, there is a constant $R = R(A, B, \lambda)$, so that all points on $\tilde f(\phi^s v)$ are of distance at most $R$ from the geodesic $\psi^t {\mathcal F}_0(v)$.
Thus two points on the same fiber of the normal projection are at most distance $2R$ apart, which gives
\[ A^{-1} l - B \leq d(f( \phi^t v), f( \phi^{t + l}v)) \leq 2R. \]
Taking $l > A(2R+B)$ guarantees $\frac{d}{ds} a_l(s, v)$ is never 0, and hence $a_l(s, v)$ is injective.
\end{proof}
\begin{prop}
For each $v \in T^1 M$, let 
\[ {\mathcal F}_l(v) = \psi^{a_l(0,v)} {\mathcal F}_0(v) \]
for $a_l$ as in Lemma \ref{alinj}.
Then ${\mathcal F}_l$ is an orbit equivalence of geodesic flows.
\end{prop}

\begin{proof}
Since ${\mathcal F}_l$ sends geodesics to geodesics, there exists a cocycle $k_l(t,v)$ so that 
${\mathcal F}_l (v) = \psi^{k_l(t,v)} {\mathcal F}_l(v)$. We need to check $t \mapsto k_l(t,v)$ is injective.
Note that
\begin{align*}
a_l(0, \phi^tv) &= \frac{1}{l} \int_0^l b(s, \phi^t v) \, ds \\
&= \frac{1}{l} \int_0^l b(s+t, v) - b(t,v) \, ds \\
&= a_l(t,v) - b(t,v).
\end{align*}
This means 
\begin{align*}
{\mathcal F}_l(\phi^t v) &= \psi^{a_l(0, \phi^t v)} {\mathcal F}_0( \phi^t v) 
= \psi^{a_l(0, \phi^t v)+ b(t,v)} {\mathcal F}_0(v) 
= \psi^{a_l(t,v)} {\mathcal F}_0(v).
\end{align*}
Therefore, 
${\mathcal F}_l(\phi^t v) = \psi^{k_l(t,v)} {\mathcal F}_l(v) = \psi^{a_l(t,v)} {\mathcal F}_0(v)$, and hence
\begin{equation}\label{kderiv}
\frac{d}{dt} |_{t=0} k_l(t,v) = \frac{d}{dt} |_{t=0} a_l(t,v) = \frac{b(l,v)}{l}.
\end{equation}
The proof of Lemma \ref{alinj} shows the above quantity is positive. So ${\mathcal F}_l$ is injective along geodesics, as desired. 
\end{proof}

We now proceed to find a H{\"o}lder estimate for $\mathcal{F}$. Most of the work is finding estimates for the map $\mathcal{F}_0$ from (\ref{F0}) (Proposition \ref{F0holder}).

\begin{lem}\label{bbd}
Let $b(t,v)$ as in (\ref{F0}). Let $A, B$ as in Lemma \ref{lem:QI}. 
Then $b(t,v)$ satisfies 
\[A^{-1} t - B' \leq b(t,v) \leq At.\] for all $t$,
where $B'$ is a constant depending only on $\lambda, A, B$. 
\end{lem}

\begin{proof}
Recall $b(t, v) = d(P_{\eta} \tilde f(p), P_{\eta} \tilde f(q))$, which is bounded above by $d(\tilde f(p), \tilde f(q))$ because orthogonal projection is a contraction in negative curvature.
This quantity is in turn bounded above by $At$, using the Lipschitz bound for $f$ in Lemma \ref{lem:QI}. 

Next, let $R$ be the constant in Lemma \ref{Morselemma}. Then $d(\tilde f(p), P_{\eta}(\tilde f(p)) \leq R$, which implies $b(t,v) \geq d(\tilde f(p), \tilde f(q)) - 2R$. The desired estimate then follows from the lower bound for $d(\tilde f(p), \tilde f(q))$ in Lemma \ref{lem:QI}.
\end{proof}

\begin{lem}\label{proj}
There is small enough $\delta_0 = \delta_0(\Lambda)$ so that for any $\delta \leq \delta_0$ the following holds. 
Fix $v \in T^1 \tilde M$ and let $x \in \tilde M$ be a point such that the orthogonal projection $P_v(x)$ of $x$ onto the bi-infinite geodesic determined by $v$ is the footpoint of $v$. 
Let $w \in W^{su}(v)$ and suppose further that $d_{su}(v, w) < \delta$. 
Then there is a constant $C$, depending only on $\Lambda$ and $d(x, p)$, 
so that $d(P_v(x), P_w(x)) < C \delta$. 
\end{lem}

\begin{proof}
Let $p$ and $q$ denote the footpoints of $v$ and $w$, respectively. 
For $u$ a vector normal to $v$, write $\gamma(s) = \exp_p(su)$. Now, given $w$, choose the unique such $u$ so that the image of $\gamma(s)$ intersects the geodesic generated by $w$.
Let $s_0$ such that $\gamma(s_0)$ intersects the geodesic determined by $w$.
%
We claim there are positive constants $\delta_0 = \delta_0(\Lambda)$ and $C = C(\Lambda)$ so that if $d_{su}(v,w) \leq \delta \leq \delta_0$ then $s_0 \leq C \delta$. 
To see this, first note that there exists $S$, depending only on $\Lambda$, such that for any $s_0 \leq S$ we have $\frac{s_0}{2} \leq \tanh(\Lambda s_0)$. 
By \cite[Proposition 4.7]{heintzehof}, we also have $\tanh(\Lambda s_0) \leq \Lambda d_{su}(v,w)$. 
Hence we can choose $\delta_0$ small enough such that $d_{su}(v,w) \leq \delta_0$ gives $s_0 \leq S$.
Hence, if $d_{su}(v,w) \leq \delta \leq \delta_0$, we get, using \cite[Proposition 4.7]{heintzehof} again, 
\[ \frac{s_0}{2} \leq  \tanh(\Lambda s_0) \leq \Lambda \delta,\]
which proves the claim.

Now let $\theta$ denote the angle between the geodesic segment $[x, \gamma(s_0)]$ and the geodesic determined by $w$. We start by showing $\theta$ is close to $\pi/2$.
In the case where $x$ and $p$ coincide, the above angle $\theta$ is the same as the angle $\theta$ in Lemma \ref{Bupper}. Thus, $\cos \theta \leq \Lambda s_0$. 

Otherwise, let $t_0 = d(x, p) \neq 0$. 
We consider two further cases: 
$d(x, \gamma(s_0)) \leq \delta$ and $d(x, \gamma(s_0)) \geq \delta$.
%
For the proof in the first case, we start by noting that
\begin{align*}
d(p, P_w(x)) &\leq d(p, \gamma(s_0)) + d(P_w(x), \gamma(s_0)) \\
&\leq d(p,q) + d(q, \gamma(s_0)) + d(P_w(x), x) + d(x, \gamma(s_0)). 
\end{align*}
Since $d(v,w) < \delta$ (by assumption), so is $d(p, q)$. 
By Lemma \ref{Bupper}, 
 $d(q, \gamma(s_0)) \leq \Lambda s_0^2 \leq C \delta^2$.
Finally, note $d(x, P_w(x)) \leq d(x, \gamma(s_0))$ by definition of $P_w(x)$. 
So applying the hypothesis $d(x, \gamma(s_0)) \leq \delta$ completes the proof in this case. 

Now we consider the case $d(x, \gamma(s_0)) \geq \delta$.
Let $v_0 \in T_x^1 M$ such that $\exp_x(t_0 v_0) = p$. 
For $0 < s \leq s_0$, let  $v(s) \in T_{x_0} M$ such that $\exp_x(t_0 v(s)) = \gamma(s)$. 
Then $X(s):= \frac{d}{dt}|_{t = t_0} \exp_x(t v(s))$ is a vector field along $\gamma(s)$. 
The hypothesis $d(x, \gamma(s_0)) \geq \delta$ allows us to bound
\begin{equation}\label{XsXs0}
\frac{\Vert X(s) \Vert}{\Vert X(s_0) \Vert} = \frac{d(x, \gamma(s))}{d(x, \gamma(s_0))} \leq 1 + \frac{s_0}{d(x, \gamma(s_0))} \leq 1 + \frac{s_0}{\delta} \leq 1 + C,
\end{equation}
where $C$ is a constant depending only on $\Lambda$. 

Let $\xi \in \partial \tilde M$ denote the common backward endpoint of $v$ and $w$.
We now claim there is a constant $C = C(t_0, \Lambda)$ 
so that 
\begin{equation}\label{costheta}
\cos \theta = \frac{\langle {\rm grad}B_{\xi}(\gamma(s_0)), X(s_0) \rangle}{\Vert X(s_0) \Vert} \leq C s_0.
\end{equation}
Since $\langle {\rm grad}B_{\xi}(\gamma(0)), X(0) \rangle = 0$, the fundamental theorem of calculus gives 
\[ \langle {\rm grad}B_{\xi}(\gamma(s_0)), X(s_0) \rangle = \int_0^{s_0} \frac{d}{ds} \langle {\rm grad}B_{\xi}(\gamma(s)), X(s) \rangle \, ds. \]
So the desired bound for $\cos(\theta)$ follows from bounding the integrand from above by $C \Vert X(s_0) \Vert$ for all $s \in [0, s_0]$. In light of (\ref{XsXs0}), it suffices to find an upper bound of the form $C \Vert X(s) \Vert$. 
To this end, we rewrite integrand using the product rule:
\begin{equation}\label{prodrule}
\frac{d}{ds} \langle {\rm grad}B(\gamma(s)), X(s) \rangle  = \langle \nabla_{\gamma'} {\rm grad}B(\gamma(s)), X(s) \rangle + \langle {\rm grad}B(\gamma(s)), \nabla_{\gamma'} X(s) \rangle.
\end{equation}
The first term on the righthand side is bounded above by
\[ \Vert X(s) \Vert  |\langle \nabla {\rm grad} B_{\gamma'}(\gamma(s)), \overline u \rangle| =  \Vert X(s) \Vert {\rm Hess}B_{\xi}(\gamma'(s), \overline u)\]
for some unit vector $\overline u$. 
Next, using that the Hessian is symmetric positive-definite bilinear form, together with Remark \ref{BupperREM}, we have 
\begin{align*}
{\rm Hess}B_{\xi}(\gamma'(s), \overline u) &\leq \frac{1}{4} {\rm Hess}B_{\xi}(\gamma'(s) + \overline u, \gamma'(s) + \overline u)
\leq \frac{\Lambda}{4} \Vert \gamma'(s) + \overline u \Vert \leq \frac{\Lambda}{2}.
\end{align*}

Now we consider the second term in (\ref{prodrule}).
First note that $\nabla_{\gamma'} X(s) = J_s'(t_0)$, where $J_s(t)$ is the Jacobi field along the geodesic $\eta_s(t) = \exp_x(t v(s))$
with initial conditions $J_s(0) = 0$ and $J_s'(0) = v(s)$. 
In order to bound $\Vert J_s'(t_0) \Vert$, we let $e_1(t) = \eta'(t), e_2(t), \dots, e_n(t)$ be a parallel orthonormal frame along $\eta(t)$. 
Let $f_1(t), \dots, f_n(t)$ such that $J_s(t) = \sum_{i=1}^n f_i(t) e_i(t)$.
The fact that $J_s$ satisfies the Jacobi equation means $f_1''(t) = 0$, so
 $f_1(t) = \langle v(s), \eta'(0) \rangle t$ and $| f_1'(t) | \leq \Vert v(s) \Vert = \Vert X(s) \Vert$. 
Now let $J_s^{\perp}$ denote the component of $J_s$ which is perpendicular to $\eta_s$. By \cite[Proposition IV.2.5]{ballmann}, we have
\[ \Vert (J_s^{\perp})' (t_0) \Vert \leq \cosh(\Lambda t_0) \Vert J_s^{\perp} (0) \Vert \leq \cosh(\Lambda t_0){\Vert X(s) \Vert}, \]
This completes the verification of (\ref{costheta}).

Now let $q'$ be the orthogonal projection of $x$ onto the geodesic determined by $w$. 
We use our bound for $\cos \theta$ to show $d(p, q')$ is small.
Consider the geodesic triangle with vertices $x$, $q'$ and $\gamma(s_0)$. 
The angle at $q'$ is $\pi/2$ by definition of orthogonal projection, and we have just shown the angle $\theta$ at $\gamma(s_0)$ satisfies $\cos \theta \leq Cs_0$, where $s_0 < C \delta$. 
Then by \cite[Theorem 7.11.2 iii)]{beardon}
%
$\tanh(d(q', \gamma(s_0)) \leq C \delta \tanh(d(x, q')) \leq C \delta$, where $C$ is the constant in (\ref{costheta}). 
Thus, for $\delta_0$ sufficiently small in terms of $C$, we see that $d(q', \gamma(s_0)) \leq 2 C \delta$ whenever $\delta < \delta_0$. 
Now recall from the first paragraph that $d(p, \gamma(s_0)) = s_0 \leq C \delta$. Noting that $d(p, q') \leq d(p, \gamma(s_0)) + d(q', \gamma(s_0))$ completes the proof. 
\end{proof}

 \begin{prop}\label{suholderest}
 Let $\delta_0 = \delta_0(\Lambda)$ be as small as in the previous lemma. 
Suppose $w \in W^{su}(v)$ and $d_{su}(v,w) < \delta_0$.
Then there is a constant $C = C(\lambda, \Lambda, A, B)$
so that
$d(\mathcal{F}_0(v), \mathcal{F}_0(w)) \leq C d_{su}(v,w)^{A^{-1} \lambda/\Lambda}$, where $A$ and $B$ are as in Lemma \ref{lem:QI}. 
The analogous statement holds if $w \in W^{ss}(v)$ instead.
 \end{prop}

\begin{proof}
Let $p$ and $q$ denote the footpoints of $v$ and $w$, respectively. 
Let $\eta_1$ be the bi-infinite geodesic in $\tilde N$ which corresponds, via the map $\overline{f}$, to the bi-infinite geodesic generated by $v$ in $\tilde M$. Then, 
by definition, $\mathcal{F}_0(v)$ is the vector tangent to $\eta_1$ based at $P_{\eta_1} (f(p))$. Similarly, $\mathcal{F}_0(w)$ is the vector tangent to $\eta_2$ based at $P_{\eta_2} (f(q))$ for the appropriate bi-infinite geodesics $\eta_2$ in $\tilde N$. In light of Lemma \ref{stableht}, it suffices to show the distance $d(P_{\eta_1}(f(p)), P_{\eta_2} (f(q))$ in $\tilde N$ is small.
By the triangle inequality,
\begin{equation}\label{projest}
d(P_{\eta_1}(f(p)), P_{\eta_2} (f(q)) \leq d(P_{\eta_1} ( f(p) ), P_{\eta_1} ( f(q) ) ) + d(P_{\eta_1} (f(q)), P_{\eta_2} f(q)).
\end{equation}

We start by estimating the first term.
Let $d_{su}(v,w) = \delta$. Then $d(p, q) < \delta$. 
By Lemma \ref{lem:QI}, we have $d(f(p), f(q)) < A \delta$. 
Since orthogonal projection is a contraction in negative curvature, the second term is bounded above by $d(f(p), f(q)) \leq A \, d(p, q)$.
 
Thus it remains to bound $d(P_{\eta_1} ( f(q) ), P_{\eta_2} ( f(q) ) )$, 
which we do by applying Lemma \ref{proj}.
Since $\mathcal{F}_0(v)$ and $\mathcal{F}_0(w)$ are on the same weak unstable leaf, there is $w'$ on the orbit of $w$ so that $\mathcal{F}_0(v)$ and $\mathcal{F}_0(w')$ are on the same strong unstable leaf. In light of Lemma \ref{proj}, it suffices to find a H{\"o}lder estimate for $d_{su}(\mathcal{F}_0(v), \mathcal{F}_0(w'))$.

Again, let $\delta = d_{su}(v,w)$ for simplicity. 
Since the unstable distance exponentially expands under the geodesic flow, there is some positive time $t$ so that $d_{su}(\phi^t v, \phi^t w) = 1$. More precisely, \cite[Proposition 4.1]{heintzehof} implies $\Lambda t \geq \log(1/\delta)$.

Next, note that $d(\mathcal{F}_0(\phi^t v), \mathcal{F}_0(\phi^t w')) \leq d(\mathcal{F}_0(\phi^t v), \mathcal{F}_0(\phi^t w)) \leq 2R + A$, where $R$ is as in Lemma \ref{Morselemma} and $A$ is the Lipschitz constant for $f$. 
Indeed, since $f$ is $A$-Lipschitz, we have $d( f(\phi^t v), f(\phi^t w)) \leq A$, and 
$d(f(\phi^t v), P_{\eta_1}f(\phi^t v)) \leq R$. 
By \cite[Theorem 4.6]{heintzehof}, we also have the bound $d_{su}(\mathcal{F}_0(\phi^t v), \mathcal{F}_0(\phi^t w')) \leq \frac{2}{\Lambda} \sinh(\Lambda(2R + A)/2)$.

By \cite[Proposition 4.1]{heintzehof}, we have the following estimate for how the unstable distance gets contracted under the geodesic flow:
\[ d_{su}(\mathcal{F}_0(v), \mathcal{F}_0(w')) \leq e^{-\lambda b(t,v)} d_{su}(\psi^{b(t,v)} \mathcal{F}_0(v), \psi^{b(t,v)} \mathcal{F}_0(w')). \]
Now recall $b(t,v) \geq A^{-1}t - B'$ from Lemma \ref{bbd}. 
This, together with the previous paragraph, gives
\[ d_{su}(\mathcal{F}_0(v), \mathcal{F}_0(w')) \leq e^{-\lambda (A^{-1}t - B')}  \frac{2}{\Lambda} \sinh(\Lambda(2R + A)/2) = C e^{- \lambda A^{-1} t}\] 
for some constant $C = C(\lambda, \Lambda, A, B)$.
Finally, we use $t \geq \frac{\log(1/\delta)}{\Lambda}$ to obtain $d_{su}(\mathcal{F}_0(v), \mathcal{F}_0(w')) \leq C \delta^{A^{-1} \lambda/\Lambda}$ for some other constant $C = C(\lambda, A, B)$.
By Lemma \ref{proj}, the second term in (\ref{projest}) is thus bounded above by $C \delta^{A^{-1} \lambda/\Lambda}$ for some other constant $C = C(\lambda, \Lambda, A, B)$, which completes the proof.
\end{proof}

 \begin{lem}\label{horcompsas}
There is small enough $\delta_0$, depending only on the curvature bounds $\lambda$ and $\Lambda$, so that if $w \in W^{ss}(v)$ and $d(v, w) < \delta_0$, then 
\[ c_1 d(v,w) \leq d_{ss}(v,w) \leq c_2 d(v,w), \]
where $c_1$ and $c_2$ are constants depending only on $\lambda$ and $\Lambda$. The analogous statement holds for $d_{su}$.  
\end{lem}

\begin{proof}
By \cite[Theorem 4.6]{heintzehof}, we have $d_{ss}(v, w) \leq \frac{\Lambda}{2} \sinh (2/ \Lambda \, d(p, q))$. 
Thus, if $d(p,q)$ is small enough (depending on $\Lambda$), we have $d_{ss}(v,w) \leq \frac{4}{\Lambda} d(p, q) \leq \frac{4}{\Lambda} d(v,w)$. 

By Lemma \ref{stableht}, $d(v,w) \leq (1 + \Lambda) d(p, q)$. By the other estimate in \cite[Theorem 4.6]{heintzehof}, there is a constant $C$, depending only on $\lambda$, so that  $d(p,q) \leq C d_{ss}(v,w)$ for all $p, q$ with $d(p, q)$ sufficiently small in terms of $\lambda$. 
\end{proof}

\begin{prop}\label{F0holder}
There exists small enough $\delta_0$, depending only on $\lambda$, $\Lambda$, 
so that for any $v, w \in T^1 \tilde M$ satisfying
$d(v,w) < \delta_0$ we have $d(\mathcal{F}_0(v), \mathcal{F}_0(w)) \leq C d(v,w)^{A^{-1} \lambda/\Lambda}$ for some constant $C = C(\lambda, \Lambda, A, B)$. 
\end{prop}
 
\begin{proof}
By Lemma \ref{temporalbd}, 
we know that for any $v, w \in T^1M$ with $d(v, w) = \delta$, 
there is a time $\sigma = \sigma(v,w) \in [- \delta, \delta]$ 
and a point $[v, w] \in T^1 \tilde M$ so that
\[ [v,w] = W^{ss}(v) \cap W^{su}(\phi^{\sigma} w). \] 
Let $\alpha = A^{-1} \lambda /\Lambda$ be the exponent from Proposition \ref{suholderest}. 
Applying Proposition \ref{suholderest}, followed by Lemma \ref{horcompsas} and Proposition \ref{localprodconst}, and finally Lemma \ref{temporalbd}, we have
\[ d( \mathcal{F}_0(v), \mathcal{F}_0([v,w])) 
\leq C' d_{ss}(v, [v,w])^{\alpha} 
\leq C d(v, \phi^{\sigma} w)^{\alpha}
\leq C (2 d(v, w))^{\alpha} \]
for some constants $C$ and $C'$ depending only on $\lambda, \Lambda, {\rm diam}(M), A$.
By a similar argument,
 \[ d(\mathcal{F}_0([v,w]), \mathcal{F}_0(\phi^{\sigma} w)) \leq C d(v,w)^{\alpha}. \]
Finally, as in the beginning of the proof of Proposition \ref{suholderest}, we have
\[ d(\mathcal{F}_0(w), \mathcal{F}_0(\phi^{\sigma} w)) \leq A \delta. \]
Now, $d(\mathcal{F}_0(v), \mathcal{F}_0(w)) \leq C d(v,w)^{\alpha}$ follows from the triangle inequality. 
 \end{proof}

 \begin{lem}\label{timetexp}
 There is a constant $C = C(\lambda, \Lambda, t)$ so that $d(\phi^t v, \phi^t w) < C d(v, w)$ for all $d(v,w) \leq \delta_0$, where $\delta_0$ depends only on $\lambda$, $\Lambda$. 
 \end{lem}
 
\begin{proof}
As before, consider $[v,w] = W^{ss}(v) \cap W^{su}(\phi^{\sigma(v,w)} w)$. 
The distance between $w$ and $\phi^{\sigma(v,w)}$ remains constant under application of $\phi^t$, and since $v$ and $[v,w]$ are on the same stable leaf, their distance contracts under application of $\phi^t$.
Finally, since $[v,w]$ and $\phi^{\sigma(v,w)} w$ are on the same strong unstable leaf, \cite[Proposition 4.1]{heintzehof}, Lemma \ref{horcompsas} and Proposition \ref{localprodconst}  imply
\[ d(\phi^t [v,w], \phi^t \phi^{\sigma(v,w)} w) \leq e^{\Lambda t} d_{su}([v,w], \phi^{\sigma(v,w)} w) \leq e^{\Lambda t} C d(v,w) \]
for some constant $C$ depending only on $\lambda$, $\Lambda$, ${\rm diam}(M)$. 
\end{proof}

\begin{lem}\label{buniholder}
Let $C$ denote the constant in Proposition \ref{F0holder}, and let $\alpha = A^{-1} \lambda /\Lambda$ denote the H{\"o}lder exponent. 
Then there is a constant $C_1 = C_1(C, t)$ so that 
\[ |b(t,v) - b(t,w)| \leq C_1 d(v,w)^{\alpha}.  \]
\end{lem}

\begin{proof}
By Proposition \ref{F0holder}, we have $d(\mathcal{F}_0(v), \mathcal{F}_0(w)) \leq C d(v,w)^{\alpha}$ and $d(\mathcal{F}_0(\phi^t v), \mathcal{F}_0(\phi^t w)) \leq C d(\phi^t v,\phi^t w)^{\alpha}$.
Applying Lemma \ref{timetexp} shows $d(\mathcal{F}_0(\phi^t v), \mathcal{F}_0(\phi^t w)) \leq C_1 d(v,w)^{\alpha}$, where $C_1$ depends on $C$ and $t$. 
The desired result now follows from Lemma \ref{timesclose}. 
\end{proof}

\begin{proof}[Proof of Proposition \ref{holder}]
We want to find a H{\"o}lder estimate for $\mathcal{F}_l(v) = \psi^{a_l(0,v)} \mathcal{F}_0(v)$,
where $a_l(0, v) = \frac{1}{l} \int_0^l b(t,v) \, dt$. 
By the triangle inequality,
\[ d(\mathcal{F}_l(v), \mathcal{F}_l(w)) \leq d(\psi^{a_l(0,v)} \mathcal{F}_0(v),\psi^{a_l(0,v)} \mathcal{F}_0(w)) + d(\psi^{a_l(0,v)} \mathcal{F}_0(w),\psi^{a_l(0,w)} \mathcal{F}_0(w)).\]

To bound the first term, note that for all $t \in [0,l]$ we have $b(t,v) \leq At \leq Al$ by Lemma \ref{bbd}. Hence the average $a_l(0,v)$ is bounded above by $Al$.  
By Lemma \ref{timetexp} and Proposition \ref{F0holder}, we have
\[ d(\psi^{a_l(0,v)} \mathcal{F}_0(v),\psi^{a_l(0,v)} \mathcal{F}_0(w)) \leq C d(\mathcal{F}_0(v), \mathcal{F}_0(w)) \leq C d(v, w)^{\alpha}, \]
where $C$ depends only on $l$ and the constant from Proposition \ref{F0holder}. As such, $C$ depends only on $\lambda$, $\Lambda$, $A$, $B$. 
By Lemma \ref{buniholder}, the second term is bounded above by 
\[ |a_l(0,v) - a_l(0, w)| \leq \frac{1}{l} \int_0^l |b(t,v) - b(t,w)| \, dt \leq C d(v,w)^{\alpha}, \]
where $C$ again depends only on $\lambda$, $\Lambda$, $A$, $B$.
\end{proof}

\bibliographystyle{amsalpha}
\bibliography{quantMLSfinite07-2025}

\end{document}